\documentclass[preprint,12pt]{elsarticle}

\usepackage{amssymb}
\usepackage{amsthm}

\newcommand{\sysn}{\left\{\begin{array}{rcl}}
\newcommand{\sysk}{\end{array}\right.}

\newtheorem{theorem}{Theorem}[section]
\newtheorem{lemma}[theorem]{Lemma}

\theoremstyle{example}
\newtheorem{example}[theorem]{Example}

\newtheorem{proposition}[theorem]{Proposition}
\theoremstyle{definition}

\newtheorem{corollary}[theorem]{Corollary}

\journal{...}

\begin{document}

\begin{frontmatter}



\title{On generalization of  theorems of Pestryakov}


\author{Alexander V. Osipov}

\ead{OAB@list.ru}


\address{Krasovskii Institute of Mathematics and Mechanics, Ural Federal
 University,

 Ural State University of Economics, Yekaterinburg, Russia}

\begin{abstract}
 In 1987 A.V. Pestryakov proved a series of theorems for cardinal functions
of the space $B_{\alpha}(X)$ of all real-valued functions of Baire
class $\alpha$ ($\alpha>0$), and he conjectured that most of these
theorems are true for  spaces containing all finite linear
combinations of characteristic functions of zero-sets in $X$. In
this paper we investigate for which theorems of Pestryakov
generalizations are valid. Also we prove some additional
propositions for function spaces applying the theory of selection
principles.
\end{abstract}

\begin{keyword}  space of Baire functions \sep
density \sep tightness \sep Lindel$\ddot{o}$f number \sep spread
\sep $G_{\delta}$-modification \sep selection principles \sep
cardinal functions



\MSC[2010] 54A25 \sep 54C35 \sep 54C30

\end{keyword}

\end{frontmatter}



\section{Introduction}

In this paper by a space we shall always mean a Tychonoff space.
Let $C_p(X)$ denote the space of continuous real-valued functions
$C(X)$ on a space
 $X$ with the topology of pointwise convergence. Let $B_0(X)=C(X)$
 and inductively define $B_{\alpha}(X)$ for each ordinal
 $\alpha\leq \omega_1$ to be the space of pointwise limits of
 sequences of functions in $\bigcup\limits_{\beta<\alpha}
 B_{\beta}(X)$. So $B_{\alpha}(X)$ a set
of all functions of Baire class $\alpha$, defined
 on a Tychonoff space $X$, provided with the pointwise convergence topology.

 The family of Baire sets of a space $X$ is the smallest family of
 sets containing the zero sets of continuous real-valued functions
 (i.e. of the form $Z(f)=\{x\in X: f(x)=0\}$), and closed under
 countable unions and countable intersections.
  The Baire sets of $X$ of multiplicative class $0$, denoted
  $Z(X)$, are the zero-sets of continuous real-valued functions.
  The sets of additive class $0$, denoted $CZ(X)$, are the
  complements of the sets in $Z(X)$.

Let $\mathcal{A}$ and $\mathcal{B}$ be sets whose elements are
families of subsets of an infinite set $X$. Then $S_1(\mathcal{A},
\mathcal{B})$ denotes the selection principle:

For each sequence $(A_n : n\in \mathbb{N})$ of elements of
$\mathcal{A}$ there is a sequence $(b_n: n\in \mathbb{N})$ such
that for each $n$, $b_n\in A_n$, and $\{b_n: n\in \mathbb{N}\}$ is
an element of $\mathcal{B}$.

The following prototype of many classical properties is called
"$\mathcal{A}$ choose $\mathcal{B}$" in \cite{ts}.

${\mathcal{A}\choose\mathcal{B}}$ : For each $A\in \mathcal{A}$
there exists $B\subset A$ such that $B\in \mathcal{B}$.

Clearly that $S_1(\mathcal{A}, \mathcal{B})$ implies
${\mathcal{A}\choose\mathcal{B}}$.

In this paper, by a cover we mean a nontrivial one, that is,
$\mathcal{U}$ is a cover of $X$ if $X=\bigcup \mathcal{U}$ and
$X\notin \mathcal{U}$.

 A cover $\mathcal{U}$ of a space $X$ is:

 $\bullet$ an {\it $\omega$-cover} if every finite subset of $X$ is contained in a
 member of $\mathcal{U}$.

$\bullet$ a {\it $\gamma$-cover} if it is infinite and each $x\in
X$ belongs to all but finitely many elements of $\mathcal{U}$.
Note that every $\gamma$-cover contains a countably
$\gamma$-cover.

For a topological space $X$ we denote:

$\bullet$ $\Omega$ --- the family of all open $\omega$-covers of
$X$;

$\bullet$ $\Gamma$ --- the family of all open $\gamma$-covers of
$X$;

$\bullet$ $Z_{\Omega}$ --- the family of all countable
$\omega$-covers of $X$ by zero-sets in $X$;

$\bullet$ $Z_{\Gamma}$ --- the family of all countable
$\gamma$-covers of $X$ by zero-sets in $X$.

\medskip

Let $(X,\tau)$ be a topological space. The {\em Baire topology}
$\tau_b$ on $X$ is the topology on the underlying set $X$ having
for a basis the family of all zero-sets of $X$. Since the
countable intersection of zero-sets is also a zero-set, it follows
that the space $X$ endowed with the Baire topology and denoted by
$X_{\aleph_0}$ is a $P$-space. Recall that a topological space is
called a {\em $P$-space} if the intersection of a countable family
of open sets is open. Let us recall also that the family of
$G_\delta$-sets in $X$ forms a base of the topology $\tau_\delta$
on $X$, and the space $X$ with the topology $\tau_\delta$ is
called  the {\em $P$-modification of} $X$ and is denoted by $PX$
or $X_\delta$ (see \cite{arpo,belspa,leri}). Clearly, $PX$ is a
$P$-space and $\tau_\delta$ is finer than the Baire topology
$\tau_b$. If $X$ is a Tychonoff space, then $X_{\aleph_0}=PX$ and
$X_{\aleph_0}$ is a Tychonoff space.  Note that the topology
$X_{\aleph_0}$ coincides with the weak topology generated by
$B_{\alpha}(X)$ for each $\alpha>0$ (\cite{jan1}).

Further, we consider the spaces $C(X_{\aleph_0})$ and
$B_{\alpha}(X)$ with the topology of pointwise convergence.

We will use the standard notation for usual cardinal invariants,
so $c$, $\chi$, $\pi\chi$, $\psi$, $w$, $\pi w$, $\psi w$, $nw$,
$d$, $t$, $l$, $s$, denote cellularity, character,
$\pi$-character, pseudocharacter, weight, $\pi$-weight,
pseudoweight, network weight, density, tightness, the
Lindel$\ddot{o}$f number, spread, respectively, see
\cite{eng,juh}. For a cardinal function $\epsilon$ denoted by $h
\epsilon(Y)=\{\epsilon(Z): Z\subseteq Y\}$,
$i\epsilon(Y)=\min\{\epsilon(Z):$ $Y$ admits a one-to-one
continuous mapping onto a space $Z \}$ and $\epsilon^*(Y)=\sup
\{\epsilon(Y^n): n\in \mathbb{N}\}$.

Since $B_{\alpha}(X)$ is dense in $\mathbb{R}^{X}$
($0\leq\alpha\leq \omega_1$), $c(B_{\alpha}(X))=\omega_0$,
$\pi\chi(B_{\alpha}(X))=\chi(B_{\alpha}(X))=\pi
w(B_{\alpha}(X))=w(B_{\alpha}(X))=|X|$.

 In 1987 A.V. Pestryakov proved the following theorems (P1-P9) for a space $B_{\alpha}(X)$ ($0<\alpha\leq \omega_1$)(\cite{pes1,pes2}).

\begin{theorem}$({\bf P1})$ $t(B_{\alpha}(X))=l^*(X_{\aleph_0})$.

\end{theorem}

\begin{theorem}$({\bf P2})$ $hd(B_{\alpha}(X))= hl^*(X_{\aleph_0})$.

\end{theorem}

\begin{theorem}$({\bf P3})$ $hl(B_{\alpha}(X))= hd^*(X_{\aleph_0})$.
\end{theorem}

\begin{theorem}$({\bf P4})$ $s(B_{\alpha}(X))= s^*(X_{\aleph_0})$.
\end{theorem}

\begin{theorem}$({\bf P5})$\label{p5} The
followings statements are equivalent.

\begin{enumerate}

\item $B_{\alpha}(X)$ is Fr$\acute{e}$chet-Urysohn;

\item $B_{\alpha}(X)$ is sequential;

\item $B_{\alpha}(X)$ is a $k$-space;

\item $B_{\alpha}(X)$ is a $\omega_1$-$k$-space;

\item $B_{\alpha}(X)$ has countable tightness;

\item $X_{\aleph_0}$ satisfies ${\Omega\choose\Gamma}$;

\item $X_{\aleph_0}$ is Lindel$\ddot{o}$f.

\end{enumerate}

\end{theorem}

\begin{theorem}$({\bf P6})$ $d(B_{\alpha}(X))=iw(X)$.

\end{theorem}

\begin{theorem}$({\bf P7})$ For $0<\alpha\leq \omega_1$,
$\psi(B_{\alpha}(X))=\psi w(B_{\alpha}(X))=i\chi
(B_{\alpha}(X))=iw(B_{\alpha}(X))=d(X_{\aleph_0})$.

\end{theorem}

\begin{theorem}$({\bf P8})$ $nw(B_{\alpha}(X))=nw(X_{\aleph_0})$.

\end{theorem}

\begin{theorem}$({\bf P9})$ $l(B_{\alpha}(X))\geq t^*(X_{\aleph_0})$.

\end{theorem}

\medskip

Put $L(X)=\{n_1\cdot f_{Z_1}+...+n_k\cdot f_{Z_k} :$ $f_{Z_i}$ is
the characteristic function of $Z_i$, $Z_i\in Z(X)$, $k,i,n_k \in
\mathbb{N}\}$. For a space $X$, define  $\mathbb{B}=\{Y : L(X)
\subseteq Y\subseteq C(X_{\aleph_0}) \}$. For example,
$B_{\alpha}(X)\in \mathbb{B}$ for each $\alpha>0$,
$C(X_{\aleph_0})\in \mathbb{B}$ and $[C_p(X)]'_{\omega_0}=\bigcup
\{cl_{\mathbb{R}^X} B : B\subset C_p(X), |B|\leq \omega_0\}\in
\mathbb{B}$.

 Pestryakov conjectured that most of theorems (P1-P9) are
true for any $\mathbb{B}(X)\in \mathbb{B}$. In this paper we check
for which Pestryakov's theorems generalizations are valid. Also we
prove some additional propositions for spaces in the class
$\mathbb{B}$.

\section{Tightness}

The following result is well known \cite{arh2}.

\begin{theorem}(Arhangel'skii-Pytkeev) $t(C_p(X))=l^*(X)$.
\end{theorem}

We prove an analogue of Theorem P1 and the Arhangel'skii-Pytkeev
Theorem for a space $\mathbb{B}(X)\in \mathbb{B}$.

\begin{theorem} \label{lem1} $t(\mathbb{B}(X))=l^*(X_{\aleph_0})$.
\end{theorem}

\begin{proof} Since $t(C_p(Y))=l^*(Y)$ for a space
$Y$ (the Arhangel'skii-Pytkeev Theorem) and
$\mathbb{B}(X)\subseteq C(X_{\aleph_0})$, then
$t(\mathbb{B}(X))\leq l^*(X_{\aleph_0})$.

Fix $n\in \mathbb{N}$. Assume that $\eta$ is an open cover of
$X^n_{\aleph_0}$. Clearly that, whenever $V\in \eta$ and
$x=(x_1,...,x_n)\in V$ there exists $W_x=\prod\limits_{i=1}^n
\{V_{x_i}$ : $V_{x_i}$ is an open in $X_{\aleph_0}$ and $x_i\in
V_{x_i}\}$ such that $x\in W_x\subset V$. Then we can consider the
cover $\mu=\{W_x : x\in X^n\}$ of $X_{\aleph_0}^n$ such that $\mu$
is a refinement of $\eta$.

For each $x=(x_1,...,x_n)\in X^n$ denote
$\widetilde{x}=\{x_1,...,x_n\}\subset X$.

Let $m\in \mathbb{N}$, $z=(z_1,...,z_m)\in X^m$.

Fix $F(z_i)\in Z(X)$ such that $z_i\in F(z_i)$ ($1\leq i \leq m$)
and if $\widetilde{x}\subset \widetilde{z}$ (i.e.
$x=(z_{i_1},...,z_{i_n})$) then $F(z_{i_k})\subset V_{z_{i_k}}$
($k=1,...,n$).

Let $f_z$ be the characteristic function of $\bigcup \{F(z_i):
1\leq i\leq m \}$. The symbol $\bf{1}$ stands for the constant
function to $1$. Note that $F=\{f_z: z\in \bigcup\{X^m: 1\leq m <
\omega \}\}\subset \mathbb{B}(X)$ and ${\bf 1}\in \overline{F}$.
Then there exists $F'\subset F$ such that ${\bf 1}\in
\overline{F'}$, $|F'|\leq \tau=t(\mathbb{B}(X))$. Then there is
$A\subset X^m$
 such that  $|A|\leq \tau$ and $F'=\{f_z\in F:
z\in A \}$. Let $Y=\{y\in X^n : \widetilde{y}\subset
\widetilde{z}, z\in A \}$. Clearly that $|Y|\leq \tau$.

 We claim that $\{W_y : y\in Y\}\subset \mu$ is a cover of $X^n$. Let
$x\in X^n$. Then $W=\{f\in \mathbb{B}(X): f(t)>0$ for all $t\in
\widetilde{x}\}$ is a neighborhood of ${\bf 1}$. There is an
$f_z\in F'\bigcap W$. We have $\widetilde{x}\subset
\bigcup\limits_{i=1}^m F(z_i)$. Let $x_k\in F(z_{i_k})$ for $1\leq
k\leq n$, $y=(z_{i_1},...,z_{i_n})$. Then $y\in Y$ and $x\in
\prod\{F(z_{i_k}): 1\leq k \leq n\}\subset \prod\{V_{y_k} : 1\leq
k \leq n\}=W_y$.

\end{proof}

\begin{corollary} $t(B_{\alpha}(X))=t([C(X)]'_{\omega_0})=t(C_p(X_{\aleph_0}))=l^*(X_{\aleph_0})$ ($\alpha>0$).
\end{corollary}

Recall that a space is said to be scattered if every nonempty
subspace of it has an isolated point.

Note that if $X$ is scattered, then $l(X)=l(X_{\aleph_0})$
\cite{leri}. Then we have the following result.

\begin{corollary} If $X$ is scattered, then $t(\mathbb{B}(X))=l^*(X)$.

\end{corollary}

Note that $l(X_{\aleph_0})=\omega_0$ implies that
$l(X^n_{\aleph_0})=\omega_0$ \cite{mir}.

\begin{corollary} $t(\mathbb{B}(X))=t(C_p(X_{\aleph_0}))=\omega_0$ if and only
if  $X_{\aleph_0}$ is Lindel$\ddot{o}$f.

\end{corollary}

\section{Hereditary density}

The following result is well known in $C_p$-theory \cite{arh2}
(for $hd(C_p(X))=\omega_0$ see \cite{vel}).

\begin{theorem} $hd(C_p(X))=hl^*(X)$.
\end{theorem}

We prove an analogue of this theorem and Theorem P2 for a space in
class $\mathbb{B}$.

\begin{theorem} $hd(\mathbb{B}(X))=hl^*(X_{\aleph_0})$.
\end{theorem}

\begin{proof} Since $hd(C_p(Y))=hl^*(Y)$ for a space $Y$ and $\mathbb{B}(X)\subseteq
C_p(X_{\aleph_0})$, then $hd(\mathbb{B}(X))\leq
hl^*(X_{\aleph_0})$.

Let $Y$ be a subspace of $X^n_{\aleph_0}$, we consider the family
$\mu=\{W_y : y\in Y\}$ of sets in $Y$ where
$W_y=\prod\limits_{i=1}^n \{Z_{y_i}$ : $Z_{y_i}$ is a zero-set in
$X$, $y_i\in Z_{y_i}\}$ such that $Y\subseteq \bigcup \mu$ and
$|\mu|=l(Y)$. Then $\gamma=\{W_y\cap Y: W_y\in \mu \}$ is an open
cover of $Y$.

Suppose $A=\{f_{Z_y} : Z_y=\bigcup\limits_{i=1}^n Z_{y_i},
W_{y}\cap Y\in \gamma$ and $f_{Z_y}$ is the characteristic
function of $Z_y\}$. Note that $A\subset \mathbb{B}(X)$. Let
$D\subset A$ such that $|D|<|A|\leq |\gamma|$. Then the family
$\eta=\{W_y\cap Y : f_{Z_y}\in D\}$ is not a cover of $Y$. Hence,
there are $W_y\in \mu$ and $y\in Y$ such that $W_y\cap Y \notin
\eta$ and $y=(y_1,...,y_n)\in W_y\cap (Y\setminus \bigcup \eta)$.
Then $\{h : |h(y_i)-f_{Z_y}(y_i)|<1, 1\leq i \leq n \}\cap
D=\emptyset$. It follows that $D$ is not dense in $A$, and
$d(A)\geq l(Y)$.

\end{proof}

\begin{corollary} $hd(B_{\alpha}(X))=hd([C(X)]'_{\omega_0})=hd(C_p(X_{\aleph_0}))=hl^*(X_{\aleph_0})$ ($\alpha>0$).
\end{corollary}

 Note that if $X$ is scattered, then
$hl(X)=hl(X_{\aleph_0})$ \cite{leri}. Then we have the following
result.

\begin{corollary} If $X$ is scattered, then $hd(\mathbb{B}(X))=hd(C_p(X_{\aleph_0}))=hl^*(X)$.

\end{corollary}

\section{Hereditary Lindel$\ddot{o}$f degree}

The following result is well known in $C_p$-theory \cite{arh2,vel}

\begin{theorem}\label{hl} If $ind(X)=0$, then $hl(C_p(X))=hd^*(X)$.
\end{theorem}

Note that $ind(X_{\aleph_0})=0$ for any space $X$. Then we have
the following theorem.

\begin{theorem} $hl(\mathbb{B}(X))=hd^*(X_{\aleph_0})$.
\end{theorem}

\begin{proof} Since $ind(X_{\aleph_0})=0$ and $\mathbb{B}(X)\subseteq
C_p(X_{\aleph_0})$, then, by Theorem \ref{hl},
$hl(\mathbb{B}(X))\leq hd^*(X_{\aleph_0})$.

First we prove an auxiliary proposition.

\begin{proposition}\label{pr1} If $Y\subset X^n$ is such that whenever
$y=(y_1,...,y_n)\in Y$ and $y_i\neq y_j$ for $i\neq j$, then
$hl(\mathbb{B}(X))\geq d(Y_{\aleph_0})$.
\end{proposition}

\begin{proof}
For each $y=(y_1,...,y_n)\in Y$ we fix a local base $\beta(y)$ at
$y$ in $X^n_{\aleph_0}$ the following of the form $\beta(y)=\{
V=\prod\limits_{i=1}^n V_i : y_i\in V_i\in Z(X), V_i\cap
V_j=\emptyset$ for $i\neq j \}$. Let
$$f_V(x)= \left\{
\begin{array}{rcl}
i, \, \,  \, \, \, \,\, \,  \, \, \, \, \,  \, \, \,\, \,\, \,  \, \, if   \, \, \, \,  x\in V_i,\\
0, \,  \, \, \, if    \, \,   x\in X\setminus \bigcup\limits_{i=1}^n V_i,\\
\end{array}
\right.
$$
and $A=\{f_V : V\in \beta(y), y\in Y\}$. Clearly that $A\subset
\mathbb{B}(X)$. Let $U(y)=\{f\in \mathbb{B}(X): |f(y_i)-i|<1\}$
and $\gamma=\{U(y): y\in Y\}$. Then $\gamma$ is a cover of $A$.
There is a subcover $\gamma'\subseteq \gamma$ such that
$|\gamma'|=l(A)\leq hl(\mathbb{B}(X))$. Consider $S=\{y\in Y:
U(y)\in \gamma'\}$. Note that $|S|\leq|\gamma'|$. We claim that
$S$ is dense in $Y_{\aleph_0}$. Fix $z=(z_1,...,z_n)\in Y$, $V\in
\beta(z)$. Then $f_V\in A$ and there is $U(y)\in \gamma'$ such
that $f_V\in U(y)$. It follows that $y\in V\cap S$ and $S$ is
dense in $Y_{\aleph_0}$.
\end{proof}

Now, by indiction on $n$, we claim that $hl(\mathbb{B}(X))\geq
hd(X^n_{\aleph_0})$.

For $n=1$ by Proposition \ref{pr1}.

Suppose that $hl(\mathbb{B}(X))\geq hd(X^k_{\aleph_0})$ for $k<n$.
Note that $X^n_{ij}=\{(x_1,...,x_n): x_i=x_j\}$ for $i\neq j$ is
homeomorphic to the space $X^{n-1}$. Set $D=\bigcup \{X^n_{ij} :
1\leq i\neq j\leq n \}$. Let $Z\subseteq X^n$. Then, by
Proposition \ref{pr1}, $d(Z_{\aleph_0}\setminus D)\leq
hl(\mathbb{B}(X))$ and, by the inductive hypothesis,
$d(Z_{\aleph_0})\leq d(Z_{\aleph_0}\setminus D)+\sum\limits_{1\leq
i\neq j\leq n} d(Z_{\aleph_0}\cap X^n_{ij})\leq
hl(\mathbb{B}(X))$.

\end{proof}

\begin{corollary}
$hl(B_{\alpha}(X))=hl([C(X)]'_{\omega_0})=hl(C_p(X_{\aleph_0}))=
hd^*(X_{\aleph_0})$ \\($\alpha>0$).
\end{corollary}

\section{Spread}

The well-known the following result in $C_p$-theory \cite{arh2}.

\begin{theorem}\label{h2} If $ind(X)=0$, then $s(C_p(X))=s^*(X)$.
\end{theorem}

Since $ind(X_{\aleph_0})=0$ for any space $X$, we have the
following theorem for the spread $s(\mathbb{B}(X))$ of a space
$\mathbb{B}(X)$ from the class $\mathbb{B}$.

\begin{theorem}
$s(\mathbb{B}(X))=s^*(X_{\aleph_0})$.
\end{theorem}

\begin{proof} Since $ind(X_{\aleph_0})=0$ and $\mathbb{B}(X)\subseteq
C_p(X_{\aleph_0})$, then, by Theorem \ref{h2},
$s(\mathbb{B}(X))\leq s^*(X_{\aleph_0})$.

First we prove an auxiliary proposition.

\begin{proposition}\label{pr21} Assume that $Y\subset X^n$ is such that whenever
$y=(y_1,...,y_n)\in Y$ and $y_i\neq y_j$ for $i\neq j$. If $Y$ is
discrete in $X^n_{\aleph_0}$, then $|Y|\leq s(\mathbb{B}(X))$.
\end{proposition}

\begin{proof}

For each $y=(y_1,...,y_n)\in Y$ we fix $V=\prod\limits_{i=1}^n
V_i$  such that $V\cap Y=\{y\}$, $y_i\in V_i\in Z(X), V_i\cap
V_j=\emptyset$ for $i\neq j$. Let
$$f_V(x)= \left\{
\begin{array}{rcl}
i, \, \,  \, \, \, \,\, \,  \, \, \, \, \,  \, \, \,\, \,\, \,  \, \, if   \, \, \, \,  x\in V_i,\\
0, \,  \, \, \, if    \, \,   x\in X\setminus \bigcup\limits_{i=1}^n V_i,\\
\end{array}
\right.
$$ and $A=\{f_V :  y\in Y\}$. Clearly that $A\subset \mathbb{B}(X)$
and $|A|=|Y|$. We claim that $A$ is discrete. If $f_U\in (f_V,
y,1)\cap A$, then $|f_U(y_i)-i|<1$ for $1\leq i \leq n$. Hence,
$y=(y_1,...,y_n)\in U$. It follows that $U=V$ and $f_U=f_V$.
\end{proof}

Now, by indiction on $n$, we claim that $s(\mathbb{B}(X))\geq
s(X^n_{\aleph_0})$.

For $n=1$ by Proposition \ref{pr21}.

Suppose that $s(\mathbb{B}(X))\geq s(X^k_{\aleph_0})$ for $k<n$.
Note that $X^n=\widetilde{X}^n\cup D$, where $D=\bigcup \{X^n_{ij}
: 1\leq i\neq j\leq n \}$, $X^n_{ij}=\{(x_1,...,x_n): x_i=x_j\}$,
$\widetilde{X}^n=X^n\setminus D$. Let $Y$ be discrete in
$X^n_{\aleph_0}$. Put $Y_1=Y\cap \widetilde{X}^n$, $Y_2=Y\cap D$,
then $Y=Y_1\cap Y_2$. By Proposition \ref{pr21}, $|Y_1|\leq
s(\mathbb{B}(X))$. If $i\neq j$, then $X^n_{ij}$ is homeomorphic
to the space $X^{n-1}$ and, hence, $|Y_2|\leq
s(X^{n-1}_{\aleph_0})$. By the inductive hypothesis, $|Y_2|\leq
s(\mathbb{B}(X))$. Note that $|Y|=|Y_1|+|Y_2|$. It follows that
$s(\mathbb{B}(X))\geq s^*(X_{\aleph_0})$.

\end{proof}

\begin{corollary}
$s(B_{\alpha}(X))=s([C(X)]'_{\omega_0})=s(C_p(X_{\aleph_0}))=
s^*(X_{\aleph_0})$ ($\alpha>0$).
\end{corollary}

\section{Modification of Theorem P5.}

Let $\kappa$ be an infinite cardinal. A space $C$ is said to be
{\it $\kappa$-initially compact} (see \cite{smir}) if every open
cover $\mathcal{V}$ of $C$ with $|\mathcal{V}|\leq \kappa$ has a
finite subcover. A space $E$ is a $\kappa$-$k$
($\kappa$-$k$-space), if whenever the subspace $A$ is non-closed
in $E$, there is a $\kappa$-initially compact subspace $C$ of $E$
with $C\cap A$ non-closed in $C$ (\cite{ger1}).

In 1982, Pytkeev \cite{py} and Gerlits \cite{ger1} independently
 proved the following result.

\begin{theorem}(Pytkeev-Gerlits) \label{th80} For a space $X$, the following are
equivalent:

\begin{enumerate}

\item $C_p(X)$ is Fr$\acute{e}$chet-Urysohn;

\item $C_p(X)$ is sequential;

\item $C_p(X)$ is a $k$-space.

\end{enumerate}

\end{theorem}

 Gerlits and Nagy defined three properties \cite{gn,ger1}:

$\bullet$ the property $(\gamma)$:  for every open $\omega$-cover
$\mathcal{V}$ of $X$ there exists a sequence $G_n\in \mathcal{V}$
such that $\{G_n : n\in \omega\}$ is a $\gamma$-cover of $X$
(${\Omega\choose\Gamma}$ in terminology of selection principles).

$\bullet$ the property $(\epsilon)$ is one of the following
equivalent properties:

(a) $X^n$ is Lindel$\ddot{o}$f for all $n\in \omega$
($l^*(X)=\omega_0$);

(b) Every open $\omega$-cover of $X$ contains a countable
$\omega$-subcover;

(c) $t(C_p(X))=\omega_0$.

$\bullet$ the property $(\varphi)$: whenever
$\mathcal{U}=\bigcup\{\mathcal{U}_n : n\in \mathbb{N} \}$ is an
open $\omega$-cover of $X$, $\mathcal{U}_n\subset
\mathcal{U}_{n+1}$ ($n\in \mathbb{N}$), there exists a sequence
$X_n\subset X$ such that $\underline{\lim} X_n=X$ and $X_n$ is
$\omega$-covered by $\mathcal{U}_n$.

Gerlits proved that a space $X$ has the property $(\gamma)$ if and
only if
 it has both $(\varphi)$ and $(\epsilon)$ (Theorem 1 in \cite{ger1}).

Let $X$ be a topological space, and $x\in X$. A subset $A$ of $X$
{\it converges} to $x$, $x=\lim A$, if $A$ is infinite, $x\notin
A$, and for each neighborhood $U$ of $x$, $A\setminus U$ is
finite. Consider the following collection:

$\bullet$ $\Omega_x=\{A\subseteq X : x\in \overline{A}\setminus
A\}$;

$\bullet$ $\Gamma_x=\{A\subseteq X : x=\lim A\}$.

\medskip

In 1982 Gerlits and Nagy \cite{gn} proved

\begin{theorem}(Gerlits-Nagy) \label{th7} For a space $X$, the following are
equivalent:

\begin{enumerate}

\item $C_p(X)$ satisfies $S_{1}(\Omega_{\bf0}, \Gamma_{\bf0})$;

\item $C_p(X)$ is Fr$\acute{e}$chet-Urysohn;

\item $X$ satisfies $S_{1}(\Omega, \Gamma)$;

\item $X$ has the property ($\gamma$), i.e. $X$ satisfies
${\Omega\choose\Gamma}$.

\end{enumerate}

\end{theorem}

In 1984, A.V. Arhangel'skii \cite{arh1} proved the following
theorem in the class of $P$-spaces.

\begin{theorem}\label{arh}(Arhangel'skii)  For a $P$-space $X$, the following
are equivalent:

\begin{enumerate}

\item $C_p(X)$ has countable tightness;

\item $C_p(X)$ is Fr$\acute{e}$chet-Urysohn;

\item $X$ is Lindel$\ddot{o}$f.

\end{enumerate}

\end{theorem}

Similar to Theorem 3 in \cite{ger1}, we get the next result.

\begin{theorem} If $\mathbb{B}(X)$ is a $\omega$-$k$-space,
then $X_{\aleph_0}$ has the property $(\varphi)$.
\end{theorem}

\begin{proof} Otherwise $\mathbb{B}(X)$ is a
$\omega$-$k$-space, yet $X$ has not the property $(\varphi)$, and
let $\mathcal{U}=\bigcup\{\mathcal{U}_n: n\in \mathbb{N}\}$
witness this. Put for $n\in \mathbb{N}$, $n\geq 1$ and $A_n=\{f\in
\mathbb{B}(X): f^{-1}(-\infty, n)$ is $\omega$-covered by
$\mathcal{U}_n\}$, $A=\bigcup\{A_n : n\in \mathbb{N}\}$. Then
$A_n$ is closed in $\mathbb{B}(X)$ for any $n$. On the other hand,
$A$ is not closed in $\mathbb{B}(X)$, because ${\bf 0}\in
\overline{A}\setminus A$. As $\mathbb{B}(X)$ is a
$\omega$-$k$-space, there is a countably compact subset $C$ of
$\mathbb{B}(X)$ such that $C\cap A$ is non-closed in $C$. As $C$
is countably compact, so also is each of its projections on the
real line: for each $x\in X$ there is an $n(x)\in \mathbb{N}$ such
that for each $f\in C$, $f(x)\leq n(x)$. Put $X_n=\{x\in X:
n(x)\leq n \}$. As the sets $X_n$ monotonically increase and their
union is $X$, we have $\underline{\lim} X_n=X$. Using now that
$\{\mathcal{U}_n\}$ witnesses that $X$ has not $(\varphi)$, we get
an $m\in \omega$ such that no $\mathcal{U}_k$ $\omega$-covers
$X_m$.

Note that $C\cap A_k=\emptyset$ if $m<k<\omega$. Indeed, let $f\in
A_k$, $m<k<\omega$. $f^{-1}(-\infty, k)$ is $\omega$-covered by
$\mathcal{U}_k$, but $X_m$ is not, so $X_m\setminus
f^{-1}(-\infty, k)\neq \emptyset$, hence, there is a point $x\in
X_m$ such that $f(x)\geq k> m$ and $n(x)\leq m$. The definition of
$X_m$ implies now that $f\notin C$.

However, this is impossible because then $C\cap A=\bigcup \{C\cap
A_k : k\leq m\}$ would be closed in $C$, contrary to the choice of
$C$.

\end{proof}

 The following theorem is proved similarly to Theorem 4 in
 \cite{ger1}; therefore, we omit the proof of this theorem.

\begin{theorem}\label{th33} If $\mathbb{B}(X)$ is a $\omega_1$-$k$-space,
then $X_{\aleph_0}$ has the property $S_1(\Omega,\Gamma)$.
\end{theorem}

Now we can consider a modification of Theorem P5.

\begin{theorem}\label{th4} Let $X$ be a Tychonoff space and $\mathbb{B}(X)\in \mathbb{B}$. Then the following
are equivalent.

\begin{enumerate}

\item $\mathbb{B}(X)$ is Fr$\acute{e}$chet-Urysohn;

\item $\mathbb{B}(X)$ is sequential;

\item $\mathbb{B}(X)$ is a $k$-space;

\item $\mathbb{B}(X)$ is a $\omega_1$-$k$-space;

\item $\mathbb{B}(X)$ has countable tightness;

\item $X_{\aleph_0}$ satisfies $S_1(\Omega,\Gamma)$;

\item $X_{\aleph_0}$ is Lindel$\ddot{o}$f.

\end{enumerate}

\end{theorem}

\begin{proof} $(1)\Rightarrow(2)\Rightarrow(3)\Rightarrow(4)$,
$(1)\Rightarrow(5)$ are immediate. Since $\mathbb{B}(X)\subseteq
C(X_{\aleph_0})$, then, by Theorem \ref{th7}, we have that
$(6)\Rightarrow(1)$ holds.

By Lemma \ref{lem1}, we have that $(5)\Rightarrow(7)$.

 By Theorem \ref{th33}, if $\mathbb{B}(X)$ is a
 $\omega_1$-$k$-space, then $X_{\aleph_0}$ satisfies
 $S_1(\Omega,\Gamma)$, i.e. $(4)\Rightarrow(6)$ holds.

$(7)\Rightarrow(1)$. Since $X_{\aleph_0}$ is a $P$-space, then, by
Theorem \ref{arh}, we have that $C_p(X_{\aleph_0})$ is
Fr$\acute{e}$chet-Urysohn. But $\mathbb{B}(X)\subseteq
C_p(X_{\aleph_0})$, hence $\mathbb{B}(X)$ is
Fr$\acute{e}$chet-Urysohn, too.

\end{proof}

\begin{corollary} Let $X$ be a Tychonoff space. Then the following
are equivalent.

\begin{enumerate}

\item $C_p(X_{\aleph_0})$ is Fr$\acute{e}$chet-Urysohn;

\item $C_p(X_{\aleph_0})$ is sequential;

\item $C_p(X_{\aleph_0})$ is a $k$-space;

 \item $C_p(X_{\aleph_0})$ is
a $\omega_1$-$k$-space;

\item $C_p(X_{\aleph_0})$ has countable tightness;

\item $X_{\aleph_0}$ satisfies $S_1(\Omega,\Gamma)$;

\item $X_{\aleph_0}$ is Lindel$\ddot{o}$f.

\end{enumerate}

\end{corollary}

\begin{corollary}\label{cor2} Let $\mathbb{B}(X)$ be a $\omega_1$-$k$-space and
$B_1(X)\subseteq \mathbb{B}(X)$. Then
$B_1(X)=\mathbb{B}(X)=C(X_{\aleph_0})$.

\end{corollary}

\begin{corollary} Assume that $X_{\aleph_0}$ satisfies $S_1(\Omega,\Gamma)$. Then
$B_1(X)=C(X_{\aleph_0})$.

\end{corollary}

\begin{corollary} Assume that $X$ is a perfectly normal space and $B_{\alpha}(X)$ is $k$-space for some $1\leq \alpha\leq \omega_1$. Then
$X$ is countable.

\end{corollary}

\begin{proposition} There exists a space $X$ such that $B_{\alpha}(X)$ is a
$\omega$-$k$-space, but not a $\omega_1$-$k$-space.
\end{proposition}

\begin{proof}
Let $X$ be the space $\omega_2\setminus L$, where $L$ denotes the
set of $\omega$-limits in $\omega_2$; then $X=X_{\aleph_0}$,
$C_p(X)=B_{\alpha}(X)$ ($0\leq \alpha \leq \omega_1$) and
$B_{\alpha}(X)$ is $\omega$-$k$ but not $\omega_1$-$k$ (see
Example in \cite{ger1}).
\end{proof}

\begin{proposition} There exists a space $X$ such that
$\omega_0=t(C_p(X))<t(B_{\alpha}(X))$ for $\alpha>0$.
\end{proposition}

\begin{proof}
The space $C_p([0,1])$ is not sequential
(Fr$\acute{e}$chet-Urysohn, $k$-space), but
$\omega_0=t(C_p([0,1]))<t(B_{\alpha}([0,1]))=\mathfrak{c}$ for any
$\alpha>0$.
\end{proof}

\begin{proposition}(MA+$\neg$ CH) There exists a set of reals $X$ such that $C_p(X)$
is sequential, but $t(\mathbb{B}(X))>\omega_0$ for any
$\mathbb{B}(X)\in \mathbb{B}$.
\end{proposition}

\begin{proof} By Theorem 1 in \cite{gm}, assuming Martin's axiom,
there exists a set of reals $X$ of cardinality the continuum such
that $X$ has the property $S_1(\Omega,\Gamma)$. Then
$X_{\aleph_0}$ is not Lindel$\ddot{o}$f and, hence, by Theorem
\ref{th4},  $t(\mathbb{B}(X))>\omega_0$ for any $\mathbb{B}(X)\in
\mathbb{B}$.
\end{proof}

\begin{theorem} If $\mathbb{B}(X)$ is a $\omega$-$k$-space,
then $X$ satisfies $S_1(Z_{\Omega},Z_{\Gamma})$.

\end{theorem}

\begin{proof} Let $\alpha=\{F_i : i\in \mathbb{N}\}$ be a $\omega$-cover
of $X$ by zero-sets of $X$. Consider $A=\{h_n: h_n=n\cdot f_n$,
$f_n$ is the characteristic function of $X\setminus F_n$, $F_n\in
\alpha$, $n\in \mathbb{N}\}$. Note that ${\bf 0}\in
\overline{A}\setminus A$. Hence, there exists a countably compact
set $C$ such that $A\bigcap C$ is not a closed subset of $C$.
Since $C$ is a countably compact set, whenever $x\in X$ there is
$n(x)\in \mathbb{N}$ such that $f(x)< n(x)$ for each $f\in C$. Let
$X_n=\{x\in X : f(x)<n$ for each $f\in C \}$. Then
$X_{n+1}\supseteq X_n$ and $X=\bigcup\limits_n X_n$.

If for every $n$ there exists $i(n)$ such that $X_n\subseteq
F_{i(n)}$, then $\{F_{i(n)} : n\in \mathbb{N}\}$ is a
$\gamma$-cover of $X$. Otherwise, there is  an $n'$ such that
$X_{n'}\setminus F_i\neq \emptyset$ for each $i\in \mathbb{N}$.
Fix an $n\in \mathbb{N}$ such that $n>n'$. There is an $x\in
X_{n'}\setminus F_n$ such that $h_n(x)=n>n'$. It follows that
$h_n\notin C$. Thus, we have that $A\bigcap C=\{h_i:
i<n'+1\}\bigcap C$ is not a closed subset of $C$, a contradiction.

\end{proof}

Recall that a space $X$ is called {\it proper analytic} if it
admits a perfect map onto an analytic subset of a complete
separable metric space. A space $X$ is { \it disjoint analytic} if
and only if it is a one-to-one continuous image of a proper
analytic space \cite{jan1}. Note that any $K$-Lusin space is a
disjoint analytic space.

\begin{theorem} Let $X$ be a disjoint analytic space and $B_1(X)\subseteq \mathbb{B}(X)$. Then the
following are equivalent:

\begin{enumerate}

\item $X$ is scattered;

\item $\mathbb{B}(X)$ is Fr$\acute{e}$chet-Urysohn.

\end{enumerate}

\end{theorem}

\begin{proof} If $X$ is scattered, then $l(X)=l(X_{\aleph_0})$
\cite{leri}. By Theorem \ref{th4}, $\mathbb{B}(X)$ is
Fr$\acute{e}$chet-Urysohn.

If $\mathbb{B}(X)$ is Fr$\acute{e}$chet-Urysohn, then, by Theorem
\ref{th4} and Corollary \ref{cor2},
$B_1(X)=\mathbb{B}(X)=C_p(X_{\aleph_0})$. Then, by Theorem 6 in
\cite{jan1}, $X$ is scattered.
\end{proof}

It is well-known that for a compact space $X$, $C_p(X)$ is
Fr$\acute{e}$chet-Urysohn if and only if $C_p(X)$ is a $k$-space
if and only if $X$ is scattered \cite{ger1,py}.

\begin{corollary} For a compact space $X$ and $\alpha>0$,
$B_{\alpha}(X)$ is Fr$\acute{e}$chet-Urysohn if and only if
$B_{\alpha}(X)$ is a $k$-space if and only if $X$ is scattered.
\end{corollary}

Thus we have that if a compact space $X$ is not scattered, then
$t(B_{\alpha}(X))\geq l(X_{\aleph_0})\geq \mathfrak{c}$.

Note that there exists a scattered space $Z$ such that
$t(B_1(Z))>\omega_0$.

\begin{example} Let $Z$ be the set of all countable ordinals endowed
with the interval topology. Then $Z$ is scattered pseudocompact
and $t(B_1(Z))>\omega_0$.
\end{example}

A.V. Arhangel'skii \cite{arh3} (see also \cite{wf}) asked the
question: For what compact spaces $X$ does the inequality
$l(X_{\aleph_0})\leq \mathfrak{c}$ hold ?

It is  well-known that the answer is positive in the following
cases:

1. $X$ is a finite product of ordered compact spaces \cite{wf}.

2. $X$ is a compact space of countable tightness \cite{py1}.

3. $X$ is a weakly Corson compact space \cite{py3}.

This implies, in particular, $t(B_{\alpha}(X))\leq \mathfrak{c}$
for any space $X$ in these classes of spaces.

In \cite{arh3,wf}, it was shown that the Lindel$\ddot{o}$f number
of $X_{\aleph_0}$ for a compact space $X$ can be arbitrary large
(for example, the Stone-$\check{C}$ech compactification $\beta(D)$
of a discrete space $D$). Therefore, the tightness of
$B_{\alpha}(X)$ for compact spaces $X$ is not bounded. E.G.
Pytkeev proved the following remarkable result (Theorem 1.1. in
\cite{py3}).

\begin{theorem}(Pytkeev) Let $X$ be a Tychonoff space. Then

$t(C_p(X))\leq t(B_{\alpha}(X))\leq exp(t(C_p(X))\cdot t(X))$.

\end{theorem}

\section{Density}

Recall that the $i$-weight $iw(X)$ of a space $X$ is the smallest
infinite cardinal number $\tau$ such that $X$ can be mapped by a
one-to-one continuous mapping onto a Tychonoff space of the weight
not greater than $\tau$.

\medskip

\begin{theorem}(Noble \cite{nob})\label{th31}   $d(C_{p}(X))=iw(X)$.
\end{theorem}

Let $A\subset Y$. Put $[A]'_{\tau}=\bigcup\{\overline{B}: B\subset
A, |B|\leq \tau \}$, $T(x,A,Y)=\min \{\tau: x\in [A]'_{\tau}\}$,
$T(A,Y)=\sup\{T(x,A,Y): x\in \overline{A}\}$. Then
$T(C_p(X),B_{\alpha}(X))=\omega_0$. Since $C_p(X)$ is dense in
$B_{\alpha}(X)$, $d(B_{\alpha}(X))\leq d(C_p(X))=iw(X)$.

Let $\mu=d(B_{\alpha}(X))$. Then there is $D\subset B_{\alpha}(X)$
such that $|D|=\mu$ and $\overline{D}=B_{\alpha}(X)$. The equality
$T(C_p(X),B_{\alpha}(X))=\omega_0$ means that
$[C_p(X)]'_{\omega_0}=B_{\alpha}(X)$. For each $d\in D$, fix a set
$C_d\subset C_p(X)$ such that $|C_d|\leq \omega_0$ and $d\in
\overline{C_d}$. Then the set $S=\bigcup \{C_d : d\in D\}$ is
dense in $C_p(X)$ and $|S|\leq \mu$. Hence, $d(B_{\alpha}(X))\geq
d(C_p(X))$. Thus, we have the Theorem P6 of Pestryakov that
$d(B_{\alpha}(X))=iw(X)$ ($0<\alpha\leq \omega_1$).

\begin{example} Let $X$ be a first-countable space such that
$|X|\leq \mathfrak{c}$ and  $iw(X)>\omega_0$. Then
$d(B_{\alpha}(X))=iw(X)>iw(X_{\aleph_0})=d(C_p(X_{\aleph_0}))$.
\end{example}

For example, if $Z$ is the set of all countable ordinals endowed
with the interval topology, then
$d(B_{\alpha}(Z))>d(C_p(Z_{\aleph_0}))$.

Note also that if $\mathfrak{c}< 2^{\omega_1}$ then
$|B_{\omega_1}(Z)|=\mathfrak{c}<2^{\omega_1}=|C_p(Z_{\aleph_0})|$,
otherwise $|B_{\omega_1}(Z)|=|C_p(Z_{\aleph_0})|$.

\section{Pseudocharacter, pseudoweight}

It is well-known that $\psi(C_p(X))=iw(C_p(X))=d(X)$ \cite{arh2}.

\begin{theorem} $\psi(\mathbb{B}(X))=\psi w(\mathbb{B}(X))=i\chi
(\mathbb{B}(X))=iw(\mathbb{B}(X))=d(X_{\aleph_0})$.
\end{theorem}

\begin{proof} Note that if there exists a condensation (one-to-one
continuous map) $f: Y \rightarrow Z$ of a space $Y$ onto a space
$Z$ then $\psi(Y)\leq \psi(Z)\leq \chi(Z)\leq w(Z)$ and
$\psi(Y)\leq \psi w(Z)\leq w(Z)$. Since the space $Z$ is
arbitrary, we get that $\psi(Y)\leq i\chi(Y)\leq iw(Y)$ and
$\psi(Y)\leq \psi w(Y)\leq iw(Y)$.

Since $iw(C_p(X))=d(X)$ (Theorem \ref{th31}) and
$\mathbb{B}(X)\subset C_p(X_{\aleph_0})$, it is enough to prove
that $d(X_{\aleph_0})\leq \psi(\mathbb{B}(X))$.

Assume that $d(X_{\aleph_0})> \psi(\mathbb{B}(X))$. Let $\{{\bf
0}\}= \bigcap \{U_\xi : \xi\in M\}$, $|M|=\psi(\mathbb{B}(X))$. We
can assume that $U_\xi=(x_1(\xi),...,x_{n}(\xi),
\epsilon(\xi))=\{f: f\in \mathbb{B}(X), |f(x_i(\xi))|<
\epsilon(\xi)\}$. Let $A=\{x_i(\xi): \xi\in M, 1\leq i \leq
n(\xi)\}$. Since $|A|<d((X_{\aleph_0})$, there exists a zero-set
$D$ in $X$ such that $D\cap A=\emptyset$. Note that the
characteristic function $\chi_D$ of the set $D$ is in $
\mathbb{B}(X)$, $\chi_D\neq {\bf 0}$ and $\chi_D\in \bigcap
\{U_\xi : \xi\in M\}$, a contradiction.
\end{proof}

\section{Network weight}

\begin{lemma}\label{lemm} Define the function $\varphi: X_{\aleph_0} \rightarrow
C_p(\mathbb{B}(X))$ by the rule: $\varphi(x)(f)=f(x)$ for each
$f\in \mathbb{B}(X)$. Then $X_{\aleph_0}$ is homeomorphic to
$\varphi(X_{\aleph_0})\subset C_p(\mathbb{B}(X))$.

\end{lemma}

\begin{proof}  Obviously, $\varphi$ is bijection from $X_{\aleph_0}$ onto
$\varphi(X_{\aleph_0})$.

 Note that $\mathbb{B}(X)\subset C_p(X_{\aleph_0})$. The equality
$\varphi^{-1}(\{h: h\in \varphi(X_{\aleph_0}),
|h(f_i)-\varphi(x)(f_i)|<\epsilon, 1\leq i \leq n, f_i\in
\mathbb{B}(X)\})=\bigcap\limits_{i=1}^n f^{-1}_i (f_i(x)-\epsilon,
f_i(x)+\epsilon)$ implies that $\varphi$ is a continuous map.

 The set $\varphi(M)=\{h: h\in \varphi(X),
|h(\chi_M)-1|<1\}$ for a characteristic function $\chi_M$ of the
zero-set $M$ is an open set in $\varphi(X)$. Thus, $\varphi^{-1}$
is a continuous map.

\end{proof}

\begin{theorem}
$nw(\mathbb{B}(X))=nw(X_{\aleph_0})$.
\end{theorem}

\begin{proof} Since $nw(C_p(Y))=nw(Y)$ for a Tychonoff space $Y$
\cite{arh2} and $\mathbb{B}(X)\subseteq C(X_{\aleph_0})$ we get
that $nw(\mathbb{B}(X))\leq nw(X_{\aleph_0})$. By Lemma
\ref{lemm}, $nw(X_{\aleph_0})\leq nw(C_p(\mathbb{B}(X))$. Thus,
$nw(X_{\aleph_0})\leq nw(\mathbb{B}(X))$.
\end{proof}

Note that $nw(X)\leq nw(X_{\aleph_0})\leq nw(X)^{\omega_0}$. Then
we have the following result.

\begin{corollary} If $\kappa=\kappa^{\omega_0}$, then
$nw(\mathbb{B}(X))=nw(C_p(X_{\aleph_0}))=nw(X)=\kappa$.
\end{corollary}

\section{The Lindel$\ddot{o}$f number}

The following result is well known in $C_p$-theory  \cite{asa}.

\begin{theorem}(Asanov)\label{th80} $l(C_p(X))\geq t^*(X)$.
\end{theorem}

For a space $\mathbb{B}(X)\in \mathbb{B}$, we have the following
result.

\begin{theorem}\label{th8} $l(\mathbb{B}(X))\geq t^*(X_{\aleph_0})$.
\end{theorem}

\begin{proof} Denote as usually $[Y]^{<\omega}$ the set of all
non-empty finite subsets of a space $Y$. Consider the topological
space $Y_p=([Y_{\aleph_0}]^{<\omega}, \tau)$ where the topology
$\tau$ generated by the base $\beta=\{H^* : H^*=\{F\in
[Y_{\aleph_0}]^{<\omega} : F\subset H\}$ for any open $H$ in
$Y\}$. Since $t(Y^n)\leq t(Y_p)$ for every $n\in \omega$
\cite{asa} it is enough to prove that $t(X_{\aleph_0 p})\leq
l(\mathbb{B}(X))$.

Let $M\subset X_{\aleph_0 p}$ and $S\in \overline{M}\setminus M$.
Note that the family $\{ <p, (-1,1)>: p\in M\}$ is a cover of the
set $\{f: f\in \mathbb{B}(X), f(S)=0\}$ where $<p, (-1,1)>=\{f:
f\in \mathbb{B}(X), f(p)\subset (-1,1)\}$. Since $\{f: f\in
\mathbb{B}(X), f(S)=0\}$ is closed in $\mathbb{B}(X)$, choose
$M'\subset M$ such that $|M'|\leq l(\mathbb{B}(X))$ and $\{ <p,
(-1,1)>: p\in M'\}$ is a cover of $\{f: f\in \mathbb{B}(X),
f(S)=0\}$. Then $S\in \overline{M'}$.

\end{proof}

Note that
$l(B_1([0,1]))=\mathfrak{c}>\omega_0=t^*([0,1]_{\aleph_0})$.

\medskip

{\bf Question.} Is it possible to replace $X_{\aleph_0}$ by $X$ in
Theorem \ref{th8} ?

\section*{Acknowledgment} The author would like to thank the referee
for careful reading and valuable comments and suggestions.





\bibliographystyle{model1a-num-names}
\bibliography{<your-bib-database>}



\end{document}